\documentclass[10pt,a4paper]{amsart}
\usepackage{geometry,lscape,multicol}                
\usepackage{graphicx}
\usepackage{amssymb,cleveref}
\usepackage{epstopdf, amsthm}
\newcommand{\Gbar}{\overline{G}}
\newcommand{\Tbar}{\overline{T}}
\newcommand{\GF}{G}

\newcommand{\B}{C}

\newtheorem{dfn}{Definition}
\newtheorem{lem}{Lemma}
\newtheorem{lemma}{Lemma}
\newtheorem{theorem}{Theorem}	
\newtheorem{cor}{Corollary}	
\newtheorem{remark}{Remark}

\makeatletter
\def\imod#1{\allowbreak\mkern7mu({\operator@font mod}\,\,#1)}
\makeatother

\title{The exact number of $r$-regular elements in finite exceptional groups}
\author{Simon Guest}
\address{School of Mathematics, University of Southampton, Southampton SO17 1BJ, UK}
\email{s.d.guest@soton.ac.uk}
\date{}                                           

\begin{document}
\maketitle

\begin{abstract}
 We calculate the precise number of $r$-regular elements in the finite exceptional groups. As a corollary we find that the proportion of $r$-regular elements is at least $3577/18432$ and for all $\epsilon>0$,  there are infinitely finite simple exceptional groups such that the proportion of $r$-regular elements is less than $3577/18432 + \epsilon$ for some prime $r$.
\end{abstract}

\section{Introduction}

If $G$ is a finite group, and $r$ is a prime, then an element of $G$ is called $r$-regular if its order is not divisible by $r$, and $r$-singular otherwise. 
Obtaining bounds (in particular lower bounds) on the number of $r$-regular elements in finite classical and finite exceptional groups is important for the design and implementation of various algorithms in computational group theory.  
In \cite{BPS}, it is shown that the proportion of $r$-regular elements in a finite simple classical group is at least $1/2d$, where $d$ is the dimension of the natural module. They show also that for every family  of finite simple exceptional groups $X(q)$, there exists a constant $c(X)\ge 1/31$ such that the proportion of $r$-regular elements in $X(q)$ is at least $c(X)$. In \cite{BCGW}, the bounds for classical groups are improved using the quokka-set method developed in \cite{NiePra} (and inspired by \cite{IKS,Leh}), and it is shown that these bounds are in some sense best possible. We use the same methods here for the finite exceptional groups. Since the Weyl group $W$ is fixed for each family $X$, and the $F$-classes of $W$ and corresponding maximal tori are known for each family of exceptional groups, it is possible to obtain exact results.
\begin{theorem} \label{main}
The precise proportions of $r$-regular elements in the finite exceptional groups are given in Tables  
\ref{2F4}---\ref{2G2}. 
\end{theorem}
\begin{cor} \label{cor:constants}
Let $X$ be a family of finite simple exceptional groups $X(q)$. The proportion of $r$-regular elements in $X(q)$ is at least $c(X)$, where $c(X)$ is given in Table \ref{constants}. Moreover, for all $\epsilon >0$, there exists infinitely many $q$ for which the proportion of $r$-regular elements in $X(q)$ is less than $c(X) + \epsilon$ for some prime $r$.
\end{cor}

\begin{table}
 \begin{tabular}{cc} 
\hline
$X(q)$ &  $c(X)$ \\ 
\hline
 ${^2}G_2(q)$& $1/2$ \\ 
 ${^2}B_2(q)$ & $1/2$ \\ 
 ${^3}D_4(q)$ & $17/48$ \\ 
 ${^2}F_4(q)$ & $7/16$\\  
 $G_2(q)$ & $11/36$\\ 
    $F_4(q)$ &  $3577/18432$\\ 
       $E_6(q)$ & $281/1296$ \\ 
          ${^2}E_6(q)$ & $281/1296$ \\
             $E_7(q)$ & $131491/589824$ \\ 
             $E_8(q)$ & $5927482903/25480396800$ \\ 
             \hline
 \end{tabular}
 \caption{Constants $c(X)$ for Corollary \ref{cor:constants}}
  \label{constants}
 \end{table}
It is also interesting to note that the proportion of odd order (that is, $2$-regular) elements can be bounded below by a constant.
\begin{cor}
If $q$ is odd then the proportion of odd order elements in the finite simple exceptional group $X(q)$ is at least   $3577/18432$. Moreover, for all $\epsilon >0$, there exist infinitely many $X(q)$ for which the proportion of odd order elements is less than $3577/18432 + \epsilon$.
\end{cor}

\section{Preliminaries}
\label{prelim}
In this section we give some useful technical results coming from algebraic group
theory and number theory. First we set up some notation. 
Let $\Gbar$ be a connected reductive algebraic group defined over $\bar{\mathbb{F}}_q$, the algebraic closure of a field $\mathbb{F}_q$ of order $q$.  Let $F$ be a Frobenius morphism of $\Gbar$, and let $G=\Gbar^F$ be the subgroup of $\Gbar$ fixed elementwise by $F$, so that $G$ is a finite group of Lie type. Let $\Tbar$ be an $F$-stable maximal torus in $\Gbar$ and let $W:=N_{\Gbar}(\Tbar)/\Tbar$ denote the Weyl group of $\Gbar$. We will say that two elements $w, w^{\prime}$ in $W$ are $F$-conjugate if there exists $x$ in $W$  such that $w^{\prime}= x^{-1}wF(x)$. This is an equivalence relation, and we will refer to the equivalence classes as \textit{$F$-classes}. Moreover, there is an explicit one-to-one correspondence between the $F$-classes of $W$ and the $\GF$-conjugacy classes of maximal tori in $\GF$.  

\subsection{The quokka set method}
\label{quokka}
Now recall that every element $g$ in $\GF$ can be expressed uniquely in the form $g=su$, where $s \in \GF$ is semisimple, $u \in \GF$ is unipotent and $su=us$. This is the multiplicative Jordan decomposition of $g$ (see \cite[p. 11]{Carter2}). We now define a quokka set as a subset of a finite group of Lie type that satisfies certain closure properties.

\begin{dfn}
{\rm Suppose that $\GF$ is a finite group of Lie type. A nonempty subset $Q$ of $\GF$ is called a \textit{quokka set}, or quokka subset of $\GF$,  if the following two conditions hold.
\begin{enumerate}
 \item[(i)] For each $g \in \GF$ with Jordan decomposition $g=su=us$, where $s$ is the semisimple part of $g$ and $u$ the unipotent part of $g$, the element $g$ is contained in $Q$ if and only if $s$ is contained in $Q$; and
\item [(ii)] the set $Q$ is a union of $\GF$-conjugacy classes.
\end{enumerate}}
\end{dfn}
For $\GF$ a finite group, and a prime $r$ not dividing $q$, define the subset
\begin{displaymath}
Q(r,\GF) := \{ g \in \GF \,: \, r\,\nmid |g| \},
\end{displaymath}
consisting of all the $r$-regular elements $g$ in $\GF$.  We see readily that $Q(r,\GF)$ is a quokka set when $G$ is a finite exceptional group. We now quote \cite[Theorem 1.3]{NiePra}, which will be our main tool.
\begin{theorem} \label{thm:maintool}
Let $\GF$, $W$, $Q(r,\GF)$ be as above. For each $F$-class $\B$ in $W$, let $T_{\B}$ denote a maximal torus corresponding to $\B$.
Then
\begin{align} \label{eqn:maintool1}
\frac{|Q(r,\GF)|}{|\GF|} &= \sum_{F-\mathrm{classes }\;\B \;\mathrm{ in } \; W}\frac{|\B|}{|W|}.\frac{|T_{\B} \cap Q(r,\GF)|}{|T_{\B}|}.
\end{align}
\end{theorem}
%

For an integer $k$ and a prime $r$, we denote by $k_r$ the $r$-part of $k$ (that is, the largest power of $r$ that divides $k$).

\begin{lem} \label{QZ}
Let $G$ be a finite simply connected exceptional group with centre $Z$. Then the proportion of $r$-regular elements in the finite simple exceptional group $G/Z(G)$ is $|Q(r,G)||Z|_r /|G|$.
\end{lem}
\begin{proof} 
  This is an easy consequence of \cite[Lemma 2.3]{BCGW}.
\end{proof}

\subsection{Cyclotomic polynomials}
\label{notheory}
Recall that the $i$th cyclotomic polynomial $\Phi_i(q)$ is defined to be the 
the unique irreducible polynomial with integer coefficients that divides $q^i-1$ and  does not a divide $q^j-1$ for any $j < i$. The $\Phi_i(q)$ satisfy the equation
\begin{equation} \label{cyclodef}
q^{i}-1 = \prod_{j | i} \Phi_j(q).
\end{equation}
The orders of maximal tori in the (simply connected) untwisted and Steinberg exceptional groups are products of cyclotomic polynomials (see \cite[p. 222]{MT} for example).  It is easy to see that the proportion of $r$-regular elements in a maximal torus $T$ is $1/|T|_r$. In view of \eqref{eqn:maintool1}, the following elementary lemma on $r$-parts of cyclotomic polynomials will therefore be useful.

\begin{lemma} \label{cyclotomic}
Let $r$ be a prime and $q$ a prime power with $(r,q)=1$. Let $e$ be the smallest positive integer such that $r | q^{e}-1$. In particular, $e$ is the multiplicative order of $q$ modulo $r$. Let $\phi_{i,r}$ be the $r$-part of the $i$th cyclotomic polynomial $\Phi_i(q)$.  If $r\ge 3$, then 
\begin{align*} 
\phi_{i,r}  = (\Phi_i(q))_r = \begin{cases}
 (q^{e}-1)_r & \text{ if $i=e$;}\\
 r & \text{ if $i=er^{f}$ and $f \ge 1$;}\\
  1& \text{ otherwise.}
 \end{cases}
\end{align*}
If $r=2$ we have 
\begin{align*} 
\phi_{i,2}= (\Phi_i(q))_2 = \begin{cases}
 (q-1)_2 & \text{ if $i=1$;}\\
  (q+1)_2 & \text{ if $i=2$;}\\
 2 & \text{ if $i=2^{f}$ and $f \ge 2$;}\\
  1& \text{ otherwise.}
 \end{cases}
\end{align*}
\begin{remark} 
 \emph{Fermat's little theorem implies that $r-1= ke$ for some integer $k$. In particular $r=ke+1 \ge e+1$.} 
\end{remark}

\end{lemma}
\begin{proof} 
This follows from \eqref{cyclodef} and the formula for the $r$-part of $q^{i}-1$ given, for example, in \cite[Lemma 2.5]{BCGW}. 
\end{proof}

\section{Proof of Theorem \ref{main}}
Recall that $e$ is the multiplicative order of $q$ modulo $r$. We denote by $\phi_{i,r}$ the $r$-part of the $i$th cyclotomic polynomial $\Phi_i(q)$ as in Lemma \ref{cyclotomic}.
The correspondence between the classes  of maximal tori and the $F$-classes of the Weyl group is well understood for the finite exceptional groups. See \cite{KSprime,Shinoda2F4,3D4classes,shoji,ShinodaF4even,MR1102992,MR898346,FleiE6E7}. For the untwisted groups, the torus orders and corresponding $F$-class representatives are obtained in \textsc{Magma} \cite{MAGMA} (using the command
\verb TwistedTorusOrders ), and the results are obtained by computer. The torus orders for ${^2}E_6(q)$ are obtained from those of $E_6(q)$ by replacing $q$ by $-q$ and the $F$-class representatives in \cite[p.98--99]{FleiE6E7} can be taken to be the same (see also \cite[Proposition 25.3]{MT}). The results for the other twisted groups can be obtained by direct calculation. We provide a detailed example in the  case $G = {^2}F_4(q)$ with $q=2^{f}$ and $f$ odd.\\
In this case, there are $11$ classes of maximal tori $T_i$. The orders $|T_i|$, together with the sizes $|C_i|$ of the corresponding $F$-classes, are given in \cite[p. 8]{Shinoda2F4}. We list them here for the convenience of the reader:
\begin{small}
\begin{multicols}{2} 
\begin{enumerate}
\item[(1)]  $|T_{1}| = (q-1)^{2}$;  $\frac{|C_1|}{|W|} = 1/16$;  
\item[(2)]  $|T_{2}| = q^{2}-1 $; $\frac{|C_2|}{|W|} = 1/4 $;
\item[(3)]  $|T_{3}| = (q-1)(q- \sqrt{2q}+1)$;  $\frac{|C_3|}{|W|} = 1/8$; 
\item[(4)]  $|T_{4}| = (q-1)(q+ \sqrt{2q}+1)$;  $\frac{|C_4|}{|W|} = 1/8$; 
\item[(5)]  $|T_{5}| = q^{2}+1 $; $\frac{|C_5|}{|W|} =1/16 $;  
\item[(6)]  $|T_{6}| = (q- \sqrt{2q}+1)^2$;  $\frac{|C_6|}{|W|} = 1/96$;  
\item[(7)]  $|T_{7}| = (q+ \sqrt{2q}+1)^2$;  $\frac{|C_7|}{|W|} =1/96 $;  
\item[(8)]  $|T_{8}| = (q+1)^{2}$;  $\frac{|C_8|}{|W|} = 1/48$; 
\item[(9)]  $|T_{9}| = q^{2}-q+1$;  $\frac{|C_9|}{|W|} = 1/6$; 
\item[(10)]  $|T_{10}| = q^{2}\!-\! \sqrt{2q^{3}} \!+\!q \!-\! \sqrt{2q}\!+\!1$;  $\frac{|C_{10}|}{|W|} = 1/12$; 
\item[(11)]  $|T_{11}| = q^{2}\!+\! \sqrt{2q^{3}} \!+\!q+\! \sqrt{2q}\!+\!1$;  $\frac{|C_{11}|}{|W|} = 1/12$.
\end{enumerate}
\end{multicols}
\end{small}
If $e=1$, then $r$ divides $q-1 = \Phi_1(q)$ and since $r \nmid q$, we have $r \ge 3$. In particular $r$ does not divide $q+1= \Phi_2(q)$. Moreover, we note that $$(q- \sqrt{2q}+1)(q+ \sqrt{2q}+1) = q^{2}+1 = \Phi_4(q)$$ so $r$ does not divide $(q\pm \sqrt{2q}+1)$ either by Lemma \ref{cyclotomic}. Furthermore, 
$$(q^{2}\!+\! \sqrt{2q^{3}} \!+\!q+\! \sqrt{2q}\!+\!1)(q^{2}\!-\! \sqrt{2q^{3}} \!+\!q \!-\! \sqrt{2q}\!+\!1) = q^{4}-q^{2}+1 = \Phi_{12}(q)$$ so $r$ does not divide $|T_{10}|$ or $|T_{11}|$. Thus we have the proportion $|Q \cap T_i|/|T_i|$ of $r$-regular elements in $T_i$ is  $1/(q-1)^2$ for $i=1$; $1/(q-1)$ for $i=2,3,4$ and $1$ for $i=5, \ldots,11$. Applying Theorem \ref{thm:maintool} implies that if $e=1$, then the proportion of $r$-regular elements in $ G={^2}F_4(q)$ is 
\begin{align*} 
  \frac{|Q|}{|G|}  &= \sum_{i=1}^{11}  \frac{|\B_i|}{|W|}.\frac{|T_{i} \cap Q|}{|T_{i}|} = \frac{1}{16(q-1)_r^{2}} \!+\! \frac{1}{(q-1)_r} \left(\frac{1}{4} + \frac{1}{8}+ \frac{1}{8}\right) +  \!\frac{1}{16}\!+\! \frac{1}{96}\! +\! \frac{1}{96}\!+\!\frac{1}{48} \!+\! \frac{1}{6} \!+\! \frac{1}{12}\!+\!\frac{1}{12} \\
  &= \frac{7}{16}+ \frac{1}{2(q-1)_r} + \frac{1}{16(q-1)_r^{2}},
\end{align*} 
which appears in line 1 of Table \ref{2F4}. If $e=2$ then $r \nmid q-1$, $\Phi_{12}(q)$ or $\Phi_4(q)$,  we have $(q^{2}-1)_r = (q+1)_r$ and $(\Phi_6(q))_r= (3,r)$. Thus $|Q \cap T_i|/|T_i| =1$ for $i=1,3,4,5,6,7,10,11$;  $1/(q+1)_r$ for $i=2$;  $1/(q+1)_r^{2}$ for $i=8$ and  $1/(3,r)$ for $i=9$. 
Applying Theorem \ref{thm:maintool} again implies that 
\begin{align*} 
  \frac{|Q|}{|G|}  &= \sum_{i=1}^{11}  \frac{|\B_i|}{|W|}.\frac{|T_{i} \cap Q|}{|T_{i}|} = 
   \!\frac{1}{16}\!+\! \frac{1}{8}\! +\! \frac{1}{8}\!+\!\frac{1}{16} \!+\! \frac{1}{96} \!+\! \frac{1}{96}\!+\!\frac{1}{12}+\!\frac{1}{12}+ 
  \frac{1}{4(q+1)_r} \!+\! \frac{1}{48(q+1)_r^{2}}+ \frac{1}{6(3,r)} \\
  &= \frac{9}{16}+ \frac{1}{6(3,r)} + \frac{1}{4(q+1)_r} + \frac{1}{48(q+1)_r^{2}},
\end{align*} 
which gives lines 2 and 3 of Table \ref{2F4}. The other cases are similar and we omit the details.
 
\section{Tables for the precise proportions in the simply connected finite exceptional groups}

We note that by Lemma \ref{QZ}, the proportion of $r$-regular elements in the finite \emph{simple} exceptional groups is obtained by multiplying by $|Z(G)|_r$.

\begin{table} 
\begin{tabular}{ll}
\hline
 $e$ & $|Q|/| {^2}F_4(q)|$ \\
 \hline
 $1$ & $7/16+ 1/2\phi_{1,r} + 1/16\phi_{1,r}^2$ \\
 $2 (r=3)$ & $89/144+ 1/4\phi_{2,r} +1/48\phi_{2,r}^{2}$\\
  $2 (r>3)$ & $35/48+1/4\phi_{2,r} +1/48\phi_{2,r}^{2}$\\
  $4$ & $77/96+3/16\phi_{4,r} + 1/96\phi_{4,r}^{2}$\\
  $6$ & $5/6 + 1/6\phi_{6,r}$ \\
  $12$ & $11/12 + 1/12\phi_{12,r}$\\
 \hline 
\end{tabular}
\caption{Proportions of $r$-regular elements in  $G= {^2}F_4(q)$ where $q$ has multiplicative order $e$ modulo $r$ 
 ($|Z(G)| = 1$)}
\label{2F4}
\end{table}


\begin{table} 
\begin{tabular}{ll}
\hline
 $e$ & $|Q|/|{^2}B_2(q)|$ \\
 \hline
 $1$ & $1/2+1/2\phi_{1,r}$ \\
 $4$ & $3/4 + 1/4\phi_{4,r}$\\
 \hline 
 \end{tabular}\caption{Proportions of $r$-regular elements in  $G= {^2}B_2(q)$ where $q$ has multiplicative order $e$ modulo $r$ ($|Z(G)| = 1$) } \label{2B2}
 \end{table} 

\begin{table} 
\begin{tabular}{ll}
\hline
 $e$ & $|Q|/|{^3}D_4(q)|$ \\
 \hline
 $1$ ($r=2, q \equiv 1 \pmod{4}$) &$(17\phi_{1,r}^2 + 12\phi_{1,r} + 4)/48\phi_{1,r}^2$\\
$1$ ($r=2, q \equiv 3 \pmod{4}$) &$(17\phi_{2,r}^2 + 12\phi_{2,r} + 4)/48\phi_{2,r}^2$\\
$1$ $(r=3)$ & $(82\phi_{1,r}^2 + 72\phi_{1,r}  + 
6)/216\phi_{1,r}^2$ \\ 
$1$ ($r>3$) & $(10\phi_{1,r}^2 + 12\phi_{1,r} + 
2)/24\phi_{1,r}^2$ \\ 
$2$ ($r=3$) & $(82\phi_{2,r}^2 + 72\phi_{2,r} + 
6)/216\phi_{2,r}^2$ \\ 
$2$ ($r>3$) & $(10\phi_{2,r}^2 + 12\phi_{2,r} + 
2)/24\phi_{2,r}^2$ \\ 
$3$ & $(15\phi_{3,r}^2 + 8\phi_{3,r} + 1)/24\phi_{3,r}^2$ \\ 
$6$ & $(15\phi_{6,r}^2 + 8\phi_{6,r} + 1)/24\phi_{6,r}^2$ \\ 
$12$ & $(3\phi_{12,r} + 1)/4\phi_{12,r}$ \\ 
\hline 
\end{tabular}
\caption{Proportions of $r$-regular elements in  $G={^3}D_4(q)$ where $q$ has multiplicative order $e$ modulo $r$  ($|Z(G)| = 1$)}
\label{3D4}
\end{table}

\begin{table} 
\begin{tabular}{ll}
\hline
$e$ & $|Q|/|F_4(q)|$ \\
\hline
$1 (r=2, q \equiv 1 \pmod{4})$ & $(3577\phi_{1,r}^4 + 3696\phi_{1,r}^3 + 1672\phi_{1,r}^2 + 192\phi_{1,r} + 16)/18432\phi_{1,r}^4$ \\ 
 $1 (r=2, q \equiv 3 \pmod{4})$ & $(3577\phi_{2,r}^4 + 3696\phi_{2,r}^3 + 1672\phi_{2,r}^2 + 192\phi_{2,r} + 16)/18432\phi_{2,r}^4$ \\
$1 (r=3)$ & $(3337\phi_{1,r}^4 + 3816\phi_{1,r}^3 + 1326\phi_{1,r}^2 + 216\phi_{1,r}
+ 9)/10368\phi_{1,r}^4$ \\  
        $1 (r>3)$ & $(385\phi_{1,r}^4 + 552\phi_{1,r}^3 + 190\phi_{1,r}^2 + 
24\phi_{1,r} + 1)/1152\phi_{1,r}^4$ \\ 
$2 (r=3)$ & $(3337\phi_{2,r}^4 + 3816\phi_{2,r}^3 + 1326\phi_{2,r}^2 + 216\phi_{2,r}
+ 9)/10368\phi_{2,r}^4$ \\  
        $2 (r>3)$ & $(385\phi_{2,r}^4 + 552\phi_{2,r}^3 + 190\phi_{2,r}^2 + 
24\phi_{2,r} + 1)/1152\phi_{2,r}^4$ \\ 
$3$ & $(55\phi_{3,r}^2 + 16\phi_{3,r} + 1)/ 72\phi_{3,r}^2$ \\ 
$4$ & $(77\phi_{4,r}^2 + 18\phi_{4,r} + 1)/ 96\phi_{4,r}^2$ \\ 
$6$ & $(55\phi_{6,r}^2 + 16\phi_{6,r} + 1)/ 72\phi_{6,r}^2$ \\ 
$8$ & $(7\phi_{8,r} + 1)/ 8\phi_{8,r}$ \\ 
$12$ & $(11\phi_{12,r} + 1)/ 12\phi_{12,r}$ \\ 
\hline
\end{tabular}
\caption{Proportions of $r$-regular elements in $G=F_4(q)$ where $q$ has multiplicative order $e$ modulo $r$  ($|Z(G)| = 1$)}
\label{F4}
\end{table}

\begin{table}
\begin{tabular}{ll}
\hline
$e$ & $|Q|/|G_2(q)|$ \\
\hline
$1 (r=2, q \equiv 1 \pmod{4})$ & $(17\phi_{1,r}^2 + 12\phi_{1,r} + 4)/48\phi_{1,r}^2$ \\ 
$1 (r=2, q \equiv 3 \pmod{4})$ & $(17\phi_{2,r}^2 + 12\phi_{2,r} + 4)/48\phi_{2,r}^2$ \\
$1 (r=3)$ & $(9\phi_{1,r}^2 + 2\phi_{1,r}^2 + 18\phi_{1,r} + 3)/ 36\phi_{1,r}^2$
\\ 
$1 (r>3)$ & $(3\phi_{1,r}^2 + 2\phi_{1,r}^2 + 6\phi_{1,r} + 1)/ 12\phi_{1,r}^2$
\\ 
$2 (r=3)$ & $(9\phi_{2,r}^2 + 2\phi_{2,r}^2 + 18\phi_{2,r} + 3)/ 36\phi_{2,r}^2$
\\ 
$2 (r >3)$ & $(3\phi_{2,r}^2 + 2\phi_{2,r}^2 + 6\phi_{2,r} + 1)/ 12\phi_{2,r}^2$
\\ 
$3$ & $(5\phi_{3,r} + 1)/ 6\phi_{3,r}$ \\ 
$6$ & $(5\phi_{6,r} + 1)/ 6\phi_{6,r}$ \\
\hline 
\end{tabular}
\caption{Proportions of $r$-regular elements in $G=G_2(q)$ where $q$ has multiplicative order $e$ modulo $r$  ($|Z(G)| = 1$) } \label{G2}
\end{table}

\begin{landscape}
 \begin{table} 
\begin{tabular}{ll}
\hline
$e$ & $|Q|/|E_6(q)|$ \\
\hline
$1$ ($r=2, q \equiv 1 \pmod{4}$) & $(179840\phi_{1,r}^6 + 131292\phi_{1,r}^5 + 113709\phi_{1,r}^4 + 
19080\phi_{1,r}^3 + 4920\phi_{1,r}^2 + 288\phi_{1,r} + 16)/829440\phi_{1,r}^6$ \\ 
$1$ ($r=2, q \equiv 3 \pmod{4}$) & $(17557\phi_{2,r}^4 + 12024\phi_{2,r}^3 + 3928\phi_{2,r}^2 + 288\phi_{2,r} + 
16)/73728\phi_{2,r}^4$ \\ 
$1 (r=3)$ & $(110240\phi_{1,r}^6 + 489348\phi_{1,r}^5 + 303003\phi_{1,r}^4 + 
71280\phi_{1,r}^3 + 9450\phi_{1,r}^2 + 972\phi_{1,r} + 27)/1399680\phi_{1,r}^6$ \\
        $1 (r=5)$ & $(61600\phi_{1,r}^6 + 90684\phi_{1,r}^5 + 
44709\phi_{1,r}^4 + 18000\phi_{1,r}^3 + 2550\phi_{1,r}^2 + 180\phi_{1,r} + 
5)/259200\phi_{1,r}^6$ \\   
        $1 (r>5)$ & $(12320\phi_{1,r}^6 + 22284\phi_{1,r}^5 + 
13089\phi_{1,r}^4 + 3600\phi_{1,r}^3 + 510\phi_{1,r}^2 + 36\phi_{1,r} + 
1)/51840\phi_{1,r}^6$ \\ 
$2 (r=3)$ & $(3337\phi_{2,r}^4 + 3816\phi_{2,r}^3 + 1326\phi_{2,r}^2 + 216\phi_{2,r}
+ 9)/10368\phi_{2,r}^4$ \\  
        $2 (r>3)$ & $(385\phi_{2,r}^4 + 552\phi_{2,r}^3 + 190\phi_{2,r}^2 + 
24\phi_{2,r} + 1)/1152\phi_{2,r}^4$ \\ 
$3$ & $(440\phi_{3,r}^3 + 183\phi_{3,r}^2 + 24\phi_{3,r} + 1)/ 648\phi_{3,r}^3$ \\ 
$4$ & $(77\phi_{4,r}^2 + 18\phi_{4,r} + 1)/ 96\phi_{4,r}^2$ \\ 
$5$ & $(4\phi_{5,r} + 1)/ 5\phi_{5,r}$ \\ 
$6$ & $(55\phi_{6,r}^2 + 16\phi_{6,r} + 1)/ 72\phi_{6,r}^2$ \\ 
$8$ & $(7\phi_{8,r} + 1)/ 8\phi_{8,r}$ \\ 
$9$ & $(8\phi_{9,r} + 1)/ 9\phi_{9,r}$ \\ 
$12$ & $(11\phi_{12,r} + 1)/ 12\phi_{12,r}$ \\ 
\hline
\end{tabular}
\caption{Proportions of $r$-regular elements in $G=E_6(q)$ where $q$ has multiplicative order $e$ modulo $r$ ($|Z(G)| = (3,q-1)$)}
\label{E6}
\end{table}
\end{landscape}

\begin{landscape}
 \begin{table}  
\begin{tabular}{ll}
\hline
$e$ & $|Q|/| {^2}E_6(q)|$ \\
\hline
$1, r=2, q \equiv 1 \pmod{4}$ & $(17557\phi_{1,r}^4 + 12024\phi_{1,r}^3 + 3928\phi_{1,r}^2 + 288\phi_{1,r} + 16)/73728\phi_{1,r}^4$ \\ 
$1, r=2 ,q \equiv 3 \pmod{4}$ & $(179840\phi_{2,r}^6 + 131292\phi_{2,r}^5 + 113709\phi_{2,r}^4 + 19080\phi_{2,r}^3 + 4920\phi_{2,r}^2 + 288\phi_{2,r} + 16)/829440\phi_{2,r}^6$ \\ 
$1 (r=3)$ & $(3337\phi_{1,r}^4 + 3816\phi_{1,r}^3 + 1326\phi_{1,r}^2 + 216\phi_{1,r} + 9)/(10368\phi_{1,r}^4)$ \\ 
$1 (r >3) $ & $(385\phi_{1,r}^4 + 552\phi_{1,r}^3 + 190\phi_{1,r}^2 + 24\phi_{1,r} + 1)/1152\phi_{1,r}^4$ \\ 
$2 (r=3)$ & $(110240\phi_{2,r}^6 + 489348\phi_{2,r}^5 + 303003\phi_{2,r}^4 + 71280\phi_{2,r}^3 + 9450\phi_{2,r}^2 + 972\phi_{2,r} + 27)/(1399680\phi_{2,r}^6)$ \\ 
$2 (r =5)$ & $(61600\phi_{2,r}^6 + 90684\phi_{2,r}^5 + 44709\phi_{2,r}^4 + 18000\phi_{2,r}^3 + 2550\phi_{2,r}^2 + 180\phi_{2,r} + 5)/259200\phi_{2,r}^6$ \\
$2 (r  > 5 ) $ & $(12320\phi_{2,r}^6 + 22284\phi_{2,r}^5 + 13089\phi_{2,r}^4 + 3600\phi_{2,r}^3 + 510\phi_{2,r}^2 + 36\phi_{2,r} + 1)/51840\phi_{2,r}^6$ \\ 
$3$ & $(55\phi_{3,r}^2 + 16\phi_{3,r} + 1)/72\phi_{3,r}^2$ \\ 
$4$ & $(77\phi_{4,r}^2 + 18\phi_{4,r} + 1)/96\phi_{4,r}^2$ \\ 
$6$ & $(440\phi_{6,r}^3 + 183\phi_{6,r}^2 + 24\phi_{6,r} + 1)/648\phi_{6,r}^3$ \\ 
$8$ & $(7\phi_{8,r} + 1)/8\phi_{8,r}$ \\ 
$10$ & $(4\phi_{10,r} + 1)/5\phi_{10,r}$ \\ 
$12$ & $(11\phi_{12,r} + 1)/12\phi_{12,r}$ \\ 
$18$ & $(8\phi_{18,r} + 1)/9\phi_{18,r}$ \\ 
\hline 
\end{tabular} 
\caption{Proportions of $r$-regular elements in  $G={^2}E_6(q)$ where $q$ has multiplicative order $e$ modulo $r$  ($|Z(G)| = (3,q+1)$)} 
\label{2E6}
\end{table}
\end{landscape}

\begin{landscape}
 \begin{table} 
\begin{tabular}{ll}
\hline
$e$ & $|Q|/|E_7(q)|$ \\
\hline
$1$ ($r=2, q \equiv 1 \pmod{4}$) & $(41419665\phi_{1,r}^7 + 95510014\phi_{1,r}^6 + 30219588\phi_{1,r}^5 +
10204152\phi_{1,r}^4 + 952560\phi_{1,r}^3 + 116256\phi_{1,r}^2 + $\\ &$ 
4032\phi_{1,r} + 
128)/371589120\phi_{1,r}^7$ \\ 
$1$ ($r=2, q \equiv 3 \pmod{4}$) & $(41419665\phi_{2,r}^7 + 95510014\phi_{2,r}^6 + 30219588\phi_{2,r}^5 +
10204152\phi_{2,r}^4 + 952560\phi_{2,r}^3 + 116256\phi_{2,r}^2 +$\\ &$
 4032\phi_{2,r} + 
128)/371589120\phi_{2,r}^7$ \\ 
$1 (r=3)$ & $(20191815\phi_{1,r}^7 + 23513057\phi_{1,r}^6 + 
12786039\phi_{1,r}^5 + 3532473\phi_{1,r}^4 + 405405\phi_{1,r}^3 +  $\\ &$ 31563\phi_{1,r}^2 
+ 1701\phi_{1,r} + 27)/78382080\phi_{1,r}^7$ \\  
 $1 (r=5)$ & $(3828825\phi_{1,r}^7 + 6047743\phi_{1,r}^6 + 
2739177\phi_{1,r}^5 + 621159\phi_{1,r}^4 + 108675\phi_{1,r}^3  $\\ &$
+ 8085\phi_{1,r}^2 + 
315\phi_{1,r} + 5)/14515200\phi_{1,r}^7$ \\   
$1 (r=7)$ & $(5360355\phi_{1,r}^7 + 7764581\phi_{1,r}^6 + 
4647699\phi_{1,r}^5 + 1140573\phi_{1,r}^4 + $\\ &$
 152145\phi_{1,r}^3 + 11319\phi_{1,r}^2 +
441\phi_{1,r} + 7)/20321280\phi_{1,r}^7$ \\ 
        $1 (r>7)$ & $(765765\phi_{1,r}^7 + 1286963\phi_{1,r}^6 + 
663957\phi_{1,r}^5 + 162939\phi_{1,r}^4 + 21735\phi_{1,r}^3 + $\\ &$
 1617\phi_{1,r}^2 + 
63\phi_{1,r} + 1)/2903040\phi_{1,r}^7$ \\ 
$2 (r=3)$ & $(20191815\phi_{2,r}^7 + 23513057\phi_{2,r}^6 + 
12786039\phi_{2,r}^5 + 3532473\phi_{2,r}^4 + 405405\phi_{2,r}^3 + $\\ &$ 31563\phi_{2,r}^2 
+ 1701\phi_{2,r} + 27)/78382080\phi_{2,r}^7$ \\  
        $2 (r=5)$ & $(3828825\phi_{2,r}^7 + 6047743\phi_{2,r}^6 + 
2739177\phi_{2,r}^5 + 621159\phi_{2,r}^4 + 108675\phi_{2,r}^3 +$\\ &$ 8085\phi_{2,r}^2 + 
315\phi_{2,r} + 5)/14515200\phi_{2,r}^7$ \\   
        $2 (r=7)$ & $(5360355\phi_{2,r}^7 + 7764581\phi_{2,r}^6 + 
4647699\phi_{2,r}^5 + 1140573\phi_{2,r}^4 + 152145\phi_{2,r}^3 + $\\ &$ 11319\phi_{2,r}^2 +
441\phi_{2,r} + 7)/20321280\phi_{2,r}^7$ \\ 
        $2 (r>7)$ & $(765765\phi_{2,r}^7 + 1286963\phi_{2,r}^6 + 
663957\phi_{2,r}^5 + 162939\phi_{2,r}^4 + 21735\phi_{2,r}^3 + $\\ &$ 1617\phi_{2,r}^2 + 
63\phi_{2,r} + 1)/2903040\phi_{2,r}^7$ \\
$3$ & $(935\phi_{3,r}^3 + 327\phi_{3,r}^2 + 33\phi_{3,r} + 1)/ 1296\phi_{3,r}^3$ \\
$4$ & $(77\phi_{4,r}^2 + 18\phi_{4,r} + 1)/ 96\phi_{4,r}^2$ \\ 
$5$ & $(9\phi_{5,r} + 1)/ 10\phi_{5,r}$ \\ 
$6$ & $(935\phi_{6,r}^3 + 327\phi_{6,r}^2 + 33\phi_{6,r} + 1)/ 1296\phi_{6,r}^3$ \\
$7$ & $(13\phi_{7,r} + 1)/ 14\phi_{7,r}$ \\ 
$8$ & $(7\phi_{8,r} + 1)/ 8\phi_{8,r}$ \\ 
$9$ & $(17\phi_{9,r} + 1)/ 18\phi_{9,r}$ \\ 
$10$ & $(9\phi_{10,r} + 1)/ 10\phi_{10,r}$ \\ 
$12$ & $(11\phi_{12,r} + 1)/ 12\phi_{12,r}$ \\ 
$14$ & $(13\phi_{14,r} + 1)/ 14\phi_{14,r}$ \\ 
$18$ & $(17\phi_{18,r} + 1)/ 18\phi_{18,r}$ \\ 
\hline
\end{tabular}
\caption{Proportions of $r$-regular elements in $G=E_7(q)$ where $q$ has multiplicative order $e$ modulo $r$ ($|Z(G)| = (2,q-1)$)}
\label{E7}
\end{table}
\end{landscape}

\begin{small}
 \begin{landscape}
 \begin{table} 
 \begin{tabular}{ll}
\hline
$e$ & $|Q|/|E_8(q)|$ \\
\hline
$1$ ($r=2, q \equiv 1 \pmod{4}$) & $(41492380321\phi_{1,r}^8 + 27525566640\phi_{1,r}^7 + 
16480551440\phi_{1,r}^6 + 2132907840\phi_{1,r}^5 + 295921248\phi_{1,r}^4 + $\\ &$
14434560\phi_{1,r}^3 + 815360\phi_{1,r}^2 + 15360\phi_{1,r} + 
256)/178362777600\phi_{1,r}^8$ \\ 
$1$ ($r=2, q \equiv 3 \pmod{4}$) & $(41492380321\phi_{2,r}^8 + 27525566640\phi_{2,r}^7 + 
16480551440\phi_{2,r}^6 + 2132907840\phi_{2,r}^5 + 295921248\phi_{2,r}^4 + $\\ &$
14434560\phi_{2,r}^3 + 815360\phi_{2,r}^2 + 15360\phi_{2,r} + 
256)/178362777600\phi_{2,r}^8$ \\
$1 (r=3)$ & $(16277566921\phi_{1,r}^8 + 21789381960\phi_{1,r}^7 + 
7567769940\phi_{1,r}^6 + 1361503080\phi_{1,r}^5 + 159928398\phi_{1,r}^4 + $\\ &$
8913240\phi_{1,r}^3 + 366660\phi_{1,r}^2 + 9720\phi_{1,r} + 
81)/56435097600\phi_{1,r}^8$ \\  
        $1 (r=5)$ & $(5363541841\phi_{1,r}^8 + 7507077000\phi_{1,r}^7 
+ 2845718900\phi_{1,r}^6 + 501215400\phi_{1,r}^5 + 53801790\phi_{1,r}^4 + $\\ &$
4095000\phi_{1,r}^3 + 150500\phi_{1,r}^2 + 3000\phi_{1,r} + 
25)/17418240000\phi_{1,r}^8$ \\   
        $1 (r=7)$ & $(1509595087\phi_{1,r}^8 + 2115252600\phi_{1,r}^7 
+ 761301260\phi_{1,r}^6 + 172854360\phi_{1,r}^5 + 18315906\phi_{1,r}^4 + $\\ &$ 
1146600\phi_{1,r}^3 + 42140\phi_{1,r}^2 + 840\phi_{1,r} + 
7)/4877107200\phi_{1,r}^8$ \\ 
        $1 (r>7)$ & $(215656441\phi_{1,r}^8 + 323507400\phi_{1,r}^7 + 
130085780\phi_{1,r}^6 + 24693480\phi_{1,r}^5 + 2616558\phi_{1,r}^4 + $\\ &$
163800\phi_{1,r}^3 + 6020\phi_{1,r}^2 + 120\phi_{1,r} + 1)/696729600\phi_{1,r}^8$ 
\\ 
$2 (r=3)$ & $(16277566921\phi_{2,r}^8 + 21789381960\phi_{2,r}^7 + 
7567769940\phi_{2,r}^6 + 1361503080\phi_{2,r}^5 + 159928398\phi_{2,r}^4 +  $\\ &$
8913240\phi_{2,r}^3 + 366660\phi_{2,r}^2 + 9720\phi_{2,r} + 
81)/56435097600\phi_{2,r}^8$ \\  
        $2 (r=5)$ & $(5363541841\phi_{2,r}^8 + 7507077000\phi_{2,r}^7 
+ 2845718900\phi_{2,r}^6 + 501215400\phi_{2,r}^5 + 53801790\phi_{2,r}^4 + $\\ &$
4095000\phi_{2,r}^3 + 150500\phi_{2,r}^2 + 3000\phi_{2,r} + 
25)/17418240000\phi_{2,r}^8$ \\   
        $2 (r=7)$ & $(1509595087\phi_{2,r}^8 + 2115252600\phi_{2,r}^7 
+ 761301260\phi_{2,r}^6 + 172854360\phi_{2,r}^5 + 18315906\phi_{2,r}^4 + $\\ &$
1146600\phi_{2,r}^3 + 42140\phi_{2,r}^2 + 840\phi_{2,r} + 
7)/4877107200\phi_{2,r}^8$ \\ 
        $2 (r>7)$ & $(215656441\phi_{2,r}^8 + 323507400\phi_{2,r}^7 + 
130085780\phi_{2,r}^6 + 24693480\phi_{2,r}^5 + 2616558\phi_{2,r}^4 + $\\ &$
163800\phi_{2,r}^3 + 6020\phi_{2,r}^2 + 120\phi_{2,r} + 1)/696729600\phi_{2,r}^8$ 
\\ 
$3$ & $(124729\phi_{3,r}^4 + 28400\phi_{3,r}^3 + 2310\phi_{3,r}^2 + 80\phi_{3,r} + 
1)/ 155520\phi_{3,r}^4$ \\ 
$4$ & $(31345\phi_{4,r}^4(5,r) + 2304\phi_{4,r}^4 + 11100\phi_{4,r}^3(5,r) + 
1270\phi_{4,r}^2(5,r) + 60\phi_{4,r}(5,r) + (5,r))/ 46080\phi_{4,r}^4(5,r)$ \\ 
$5$ & $(551\phi_{5,r}^2 + 48\phi_{5,r} + 1)/ 600\phi_{5,r}^2$ \\ 
$6$ & $(124729\phi_{6,r}^4 + 28400\phi_{6,r}^3 + 2310\phi_{6,r}^2 + 80\phi_{6,r} + 
1)/ 155520\phi_{6,r}^4$ \\ 
$7$ & $(13\phi_{7,r} + 1)/ 14\phi_{7,r}$ \\ 
$8$ & $(161\phi_{8,r}^2 + 30\phi_{8,r} + 1)/ 192\phi_{8,r}^2$ \\ 
$9$ & $(17\phi_{9,r} + 1)/ 18\phi_{9,r}$ \\ 
$10$ & $(551\phi_{10,r}^2 + 48\phi_{10,r} + 1)/ 600\phi_{10,r}^2$ \\ 
$12$ & $(253\phi_{12,r}^2 + 34\phi_{12,r} + 1)/ 288\phi_{12,r}^2$ \\
$14$ & $(13\phi_{14,r} + 1)/14\phi_{14,r}$ \\ 
$15$ & $(29\phi_{15,r} + 1)/ 30\phi_{15,r}$ \\ 
$18$ & $(17\phi_{18,r} + 1)/ 18\phi_{18,r}$ \\ 
$20$ & $(19\phi_{20,r} + 1)/ 20\phi_{20,r}$ \\ 
$24$ & $(23\phi_{24,r} + 1)/ 24\phi_{24,r}$ \\ 
$30$ & $(29\phi_{30,r} + 1)/ 30\phi_{30,r}$ \\
\hline
\end{tabular}
\caption{Proportions of $r$-regular elements in $G=E_8(q)$ where $q$ has multiplicative order $e$ modulo $r$ ($|Z(G)| = 1$)}
\label{E8}
\end{table}
\end{landscape}
\end{small}

%

\begin{table} 
\begin{tabular}{ll}
\hline
 $e$ & $|Q|/|{^2}G_2(q)|$ \\
 \hline
 $1$ ($r=2$) & $7/12+ 1/6 (q+1)_2$\\
 $1$ ($r>3$) & $1/2+1/2\phi_{1,r}$\\
 $2$ & $5/6 + 1/6\phi_{2,r}$ \\
 $6$ & $5/6 + 1/6\phi_{6,r}$\\
 \hline
\end{tabular} 
\caption{Proportions of $r$-regular elements in $G={^2}G_2(q)$ where $q$ has multiplicative order $e$ modulo $r$  ($|Z(G)| = 1$)}
\label{2G2}
\end{table}


\end{document}